\documentclass{amsart}
\usepackage{amsmath,amsfonts,amsthm,amssymb,indentfirst,epic,xurl,graphics,needspace}

\newtheorem{theorem}{Theorem}
\newtheorem{lemma}[theorem]{Lemma}

\newtheorem{proposition}[theorem]{Proposition}
\newtheorem{definition}[theorem]{Definition}
\newtheorem{question}[theorem]{Question}

\numberwithin{equation}{section}
\numberwithin{theorem}{section}

\newcommand{\Z}{\mathbb{Z}}
\renewcommand{\r}{\mathrm}

\begin{document}

\begin{center}
\texttt{Comments, corrections,
and related references welcomed, as always!}\\[.5em]
{\TeX}ed \today
\\[.5em]
\vspace{2em}
\end{center}

\title%
{Strongly clean ring elements that are one-sided inverses}
\thanks{%
Readable at \url{http://math.berkeley.edu/~gbergman/papers/}.
After publication, any updates, errata, related references,
etc., found will be noted at
\url{http://math.berkeley.edu/~gbergman/papers/abstracts}\,.\\
}


\author{George M.\ Bergman}
\address{Department of Mathematics\\
University of California\\
Berkeley, CA 94720-3840, USA}
\email{gbergman@math.berkeley.edu}

\begin{abstract}
A longstanding open question is whether every strongly clean ring
is Dedekind-finite (definitions recalled below).
We give an example of a ring with two strongly clean elements
that are one-sided, but not two-sided inverses of one
another, suggesting that the answer to that question may be negative.

We discuss possible ways of strengthening this result
to give a full negative answer.
We end with some brief observations on related topics, in particular,
uniquely strongly clean rings.
\end{abstract}
\maketitle

\section{The example}\label{S.Example}

Here is some standard usage that we will follow:

\begin{definition}\label{D.gen}
In this note, all rings are understood to be associative and unital.

An element $x$ of a ring $R$ is
called {\em clean} if it can be written as the sum $e+u$ of
an idempotent $e$ and a unit $u$ of $R,$ and {\em strongly clean}
if it can be so written with $e$ and $u$ commuting with each other,
equivalently, commuting with $x.$
A ring $R$ is called clean, respectively strongly
clean, if all of its elements have the property named
{\rm(}\cite{WKN}, \cite{Y+L}{\rm)}.

A ring $R$ is called {\em Dedekind-finite} if every
one-sided-invertible element of $R$ is invertible.
\end{definition}

The question of whether every strongly clean ring is
Dedekind-finite was posed in
1999 by W.\,K.~Nicholson \cite[Question~2, p.\,3590]{WKN}.
For $k$ a field, we construct below a $\!k\!$-algebra $R$ having
a pair of elements that are both strongly clean, and
which are one-sided but not two-sided inverses of one another.
Thus, if this ring can somehow be embedded in a ring
where {\em all} elements are strongly clean, the
resulting example would answer Nicholson's question in the negative.

Here is some notation that will be used in describing our example:

\begin{definition}\label{D.kVR}
Throughout this section $k$ will be a field, and $V$ will denote
the underlying k-vector-space of the subalgebra of the rational
function field $k(t)$ comprising elements having
denominator not divisible by $t;$
i.e., of the localization $k[t]_{(t)}$
of the polynomial ring $k[t]$ at the prime ideal~$(t).$
We denote by $R$ the $\!k\!$-algebra $\r{End}_k(V)$ of all
$\!k\!$-linear endomaps of the vector space $V,$ written on
the left of their arguments and composed accordingly.
\end{definition}

(We could equally well take for $R$ the
subalgebra of $\r{End}_k(V)$ generated by the five maps
that will be denoted
$y,$ $x,$ $(y-1)^{-1},$ $e,$ and $(x-e)^{-1}$ below.
But it will be convenient to have a ring that we can refer to when
defining these elements.)

Let $y$ be the endomorphism of $V$ given by multiplication by $t$
in $k[t]_{(t)}$:
\begin{equation}\begin{minipage}[c]{35pc}\label{d.y=}
$y(f(t))\ =\ t\,f(t)$ \ for $f(t)\in V,$
\end{minipage}\end{equation}
and $x$ the endomorphism of $V$ given by multiplying by $t^{-1}$ and,
in power series notation, dropping the $t^{-1}$ term if any.
In rational function notation, this is
\begin{equation}\begin{minipage}[c]{35pc}\label{d.x=}
$x(f(t))\ =\ t^{-1} (f(t)-f(0))$ \ for $f(t)\in V.$
\end{minipage}\end{equation}

It is clear that
\begin{equation}\begin{minipage}[c]{35pc}\label{d.xy}
$x\,y\ =\ 1.$
\end{minipage}\end{equation}

On the other hand, $y\,x$ has the effect
of dropping the $t^0$ term from the power series representing
an element of $V,$ so
\begin{equation}\begin{minipage}[c]{35pc}\label{d.yx}
$y\,x\ \neq\ 1.$
\end{minipage}\end{equation}

We will now prove $x$ and $y$ strongly clean in $R.$

Note that $y-1$ is invertible in $R$:~ it takes $f(t)$ to $(t-1)f(t),$
and $t-1$ is invertible in $k[t]_{(t)}.$
Thus $y$ is the sum of the invertible element $y-1$ and
the idempotent element $1,$ which commute, showing that
\begin{equation}\begin{minipage}[c]{35pc}\label{d.y_cln}
$y$ \,is strongly clean in\, $R.$
\end{minipage}\end{equation}

In proving $x$ strongly clean, we can't take for $e$
one of the idempotents $1$ or $0$:
$x-1$ is not invertible since $x$ fixes $(1-t)^{-1}$ (easily seen
by looking at power series expansions);
and $x-0$ is not invertible since $x$ annihilates $1.$
So to get a decomposition $x=e+u,$
we will have to use a nontrivial idempotent~$e.$

To construct $e,$ note that
\begin{equation}\begin{minipage}[c]{35pc}\label{d.oplus}
$V \,= \ V_0\oplus V_1,$
\end{minipage}\end{equation}
where
\begin{equation}\begin{minipage}[c]{35pc}\label{d.V0}
$V_0$ \ is the subspace \ $k[t]\,\subseteq\,V,$
\end{minipage}\end{equation}
and
\begin{equation}\begin{minipage}[c]{35pc}\label{d.V1}
$V_1\subseteq V$ is the space of rational functions
whose numerators have lower degrees than their denominators.
\end{minipage}\end{equation}
Indeed, given an element $p(t)/q(t)\in V,$ where $p(t),\,q(t)\in k[t],$
its decomposition as in~\eqref{d.oplus} arises from the
decomposition of $p(t)$ as a multiple of $q(t)$ in $k[t],$
plus a remainder of degree less than that of $q(t).$
(Intuitively, $V_1$ is the set of $f(t)\in V$
such that $f(\infty) = 0.)$
We now
\begin{equation}\begin{minipage}[c]{35pc}\label{d.e}
Let $e\in R$ be the projection of $V$ onto $V_0$
under the decomposition~\eqref{d.oplus}.
\end{minipage}\end{equation}

Clearly, $x,$ defined in~\eqref{d.x=}, carries $V_0$ into itself.
A little thought shows that it also carries $V_1$ into itself:
Although when~\eqref{d.x=} is applied to an element $f\in V_1,$ the
expression $f(t) - f(0)$ has in general lost the property that $f$
had, of having value $0$ at $\infty,$ it still has a finite value there,
namely $-f(0);$ so multiplying by $t^{-1}$ again brings the value
at $\infty$ to $0.$
(The reader can translate all this into reasoning
about degrees and values at $0$ of numerators and denominators.)

Hence,
\begin{equation}\begin{minipage}[c]{35pc}\label{d.xe}
$x$ commutes with $e.$
\end{minipage}\end{equation}

Now since $V_1$ contains no nonzero constants, we see
from~\eqref{d.x=} that $x$ has trivial kernel on $V_1.$
It is also surjective on $V_1$: Given $f(t)\in V_1,$ if
we let $c\in k$ be
the value of $t f(t)$ at $\infty,$ we find that $t f(t) - c\in V_1$
and $x$ carries that element to $f(t).$
(This is like the construction of the
preceding paragraph, with the roles of $0$ and $\infty$ interchanged.)
Thus $x$ acts invertibly on $V_1.$
Since $e$ annihilates $V_1,$ this tells us that
\begin{equation}\begin{minipage}[c]{35pc}\label{d.x-e_V1}
$x-e$ acts invertibly on $V_1.$
\end{minipage}\end{equation}

How does $x-e$ behave on $V_0$?
For every $n\geq 0,$ we see from~\eqref{d.x=} that
the action of $x$ on polynomials
in $t$ of degree $\leq n$ is nilpotent, while by~\eqref{d.e},
$e$ acts as the identity endomorphism on $V_0;$ so
\begin{equation}\begin{minipage}[c]{35pc}\label{d.x-e_V0}
$x-e$ acts invertibly on $V_0.$
\end{minipage}\end{equation}

Together, \eqref{d.x-e_V1} and~\eqref{d.x-e_V0}
say that $x-e$ acts invertibly on $V;$ so $x$ is the sum
of the idempotent $e$ and the invertible element $x-e,$
which commute by~\eqref{d.xe}.
So
\begin{equation}\begin{minipage}[c]{35pc}\label{d.str_cl}
$x$ is strongly clean.
\end{minipage}\end{equation}

We have thus proved

\begin{proposition}\label{P.yx_neq_1}
In the k-algebra $R$ of all $\!k\!$-vector-space endomorphisms of
the vector space $V$ defined in Definition~\ref{D.kVR},
the elements $x$ and $y$ described
by~\eqref{d.y=} and~\eqref{d.x=} are strongly clean,
and satisfy \mbox{ $x\,y = 1$} but $y\,x \neq~1.$\qed
\end{proposition}

\section{Thoughts on possible stronger results}\label{S.thoughts}

I wonder whether one can push the approach of the above
construction further, and get an
example actually answering Nicholson's question:

\begin{question}\label{Q.more}
As in Definition~\ref{D.kVR},
let $k$ be a field and $V$ the underlying $\!k\!$-vector-space
of $k[t]_{(t)};$ and as in~\eqref{d.oplus}-\eqref{d.V1},
let $V=V_0\oplus V_1,$ where $V_0=k[t]$ and $V_1$
is the space of elements of $k[t]_{(t)}$
whose numerators have lower degrees than their denominators.
Further,
\begin{equation}\begin{minipage}[c]{35pc}\label{d.R=}
Let $R$ be the algebra of $\!k\!$-linear
endomorphisms $x$ of $V$ such that there exist
subspaces $V'_0\subseteq V_0$ and $V'_1\subseteq V_1,$
each of finite $\!k\!$-codimension in the
indicated space, and elements $r_0,\,r_1\in k(t),$ such that
$x$ carries $V'_0$ into $V_0$ by multiplication by $r_0,$
and carries $V'_1$ into $V_1$ by multiplication by $r_1.$
\end{minipage}\end{equation}
Is $R$ strongly clean?
\end{question}

Note that if an element $x$ satisfies the condition of~\eqref{d.R=},
the elements $r_0$ and $r_1$ referred to will be uniquely
determined by $x,$ but the subspaces $V'_0$ and $V'_1$
will not, though there will, of course, be maximal
subspaces on which $x$ acts as multiplication
by the indicated elements.
Given $V'_0$ and $V'_1,$
the element $x$ is uniquely determined by $r_0,$ $r_1,$ and
the behavior of $x$ on vector-space complements of $V'_0$ in $V_0$
and of $V'_1$ in $V_1.$
On those complements, $x$ is allowed to act as an arbitrary
$\!k\!$-linear map into $V;$ it is not required to
carry elements of $V_0$ or $V_1$ into that same subspace,
nor to respect multiplication by members of $k[t]_{(t)}.$
(And indeed, the $y$ of~\eqref{d.y=}
carries $(1-t)^{-1}\in V_1$ to $t\,(1-t)^{-1},$ whose
numerator is not of lower degree than its denominator,
as would be required for it to belong to $V_1.$
Its $V_0\oplus V_1$ decomposition is $-1+(1-t)^{-1}.)$

It is easy to see that the set of maps $x$
satisfying the condition of~\eqref{d.R=}
is closed under addition, and not hard to check
that it is closed under composition.

In studying Question~\ref{Q.more}, it might help
to look at a more general construction of which~\eqref{d.R=} is a
particular case.
If $V$ is an infinite-dimensional vector space over
a field $k,$ suppose we define a ``partial endomorphism''
of $V$ to mean a $\!k\!$-linear map from a subspace of finite
codimension in $V$ into $V,$ and a ``quasi-endomorphism''
of $V$ to mean an equivalence class of partial endomorphisms
under the equivalence relation of agreeing on a subspace
of finite codimension.
({\em Have these concepts been studied?
If so, what are they called?})
The quasi-endomorphisms of $V$ form a $\!k\!$-algebra in
a natural way.
Now suppose $V_0$ and $V_1$ are infinite-dimensional vector
spaces over $k,$ and $D_0,$ $D_1$ are division subalgebras of the
algebras of quasi-endomorphisms of $V_0$ and $V_1$ respectively.
Let $R$ be the ring of all endomorphisms
of $V_0\times V_1$ which, when restricted to each $V_i,$
carry a subspace of finite codimension back into $V_i,$
and such that the induced quasi-endomorphism is a member of~$D_i.$
(Thus, $R$ has a natural homomorphism to $D_0\times D_1,$
whose kernel is the ideal of endomorphisms of $V$ of finite rank.)

Sadly, my best guess is that the answers to
Question~\ref{Q.more} and its generalization to algebras
of the sort described in the preceding paragraph are,
at least as presently formulated, negative.
This is based on the following observation.
We have used two spaces $V_0$ and $V_1$ in Question~\ref{Q.more}
only to set up a framework in which to construct pairs of elements
which are one-sided but not two-sided inverses
to one another; but the strong cleanness
conclusion, if it holds, should apply to the corresponding
construction on a single vector space.
Now let $V$ be the underlying $\!k\!$-vector space
of the polynomial ring $k[t],$ which has, as above, a natural
algebra of quasi-endomorphisms isomorphic to $k(t),$
and let $x\in\r{End}(V)$ be the
operation of multiplication by $t.$
Then given any $y\in\r{End}(V)$ centralizing $x,$
we can regard $y(1)$ as a polynomial $p(t),$ and
it is easy to verify that $y$ must act on all of $V$ as
multiplication by $p(t);$ so the centralizer of $x$ is
isomorphic as a $\!k\!$-algebra to $k[t],$ and $x$ is
not clean in that $\!k\!$-algebra.
(Restricting our search for endomorphisms centralizing $x$
to the subalgebra of elements of $\r{End}(V)$ which induce
quasi-endomorphisms belonging to $k(t)$
obviously does not improve things.)

On the other hand, if we let $V$ be the underlying
vector space of $k[t]_{(t)}$ or of $k[[t]],$ then
the endomorphism corresponding to
multiplication by $t$ does become strongly
clean in that over-ring; indeed, again calling that
operation $x,$ we see that $x-1$ is invertible in
those two algebras.
So perhaps some variant of the idea
of Question~\ref{Q.more} will give a strongly
clean ring containing elements $x$ and $y$
as in Proposition~\ref{P.yx_neq_1}.

Perhaps a much stronger sort of result is true than what is
asked for above:

\begin{question}\label{Q.still_more}
For every {\rm(}associative unital{\rm)} algebra $R$ over
a field $k,$ and every $x\in R,$ does there exist
a $\!k\!$-algebra $R'\supseteq R$ in which $x$ is strongly clean?

{\rm(}If so, then by a transfinite induction one can embed
every $\!k\!$-algebra $R$ in a strongly clean $\!k\!$-algebra.{\rm)}
\end{question}

To prove such a result, one might start
with an arbitrary faithful action of $R$ on a vector space $V,$
and try to obtain induced faithful actions on new
vector spaces that would make the action of $x$ strongly clean
in the endomorphism rings of these spaces.
I have looked at the case where $W=V^*,$ the vector-space dual of $V$
(with the induced action of $R$ on $W$ written on the right if we have
been writing the action on $V$ on the left, to handle the
contravariance), but not gotten anywhere.
Variants which I haven't looked at would be to let $W$ be
an infinite direct sum or direct product of copies of $V,$ or
an ultrapower of $V.$

A different approach to Question~\ref{Q.still_more}, which I have also
attempted without success, but which others might try is, assuming a
$\!k\!$-algebra $R$ and an element $x\in R$ given, to
see what happens when one adjoins to $R$ two elements
$u$ and $u^{-1},$ universal for
satisfying the five relations saying that they
are inverses to each other, that they commute with $x,$ and
that $x-u$ is idempotent.
A slight variant would be to first adjoin a universal
idempotent $e$ commuting with $x,$ study what one can say
about the structure of the resulting ring, {\em then}
adjoin a universal inverse to $x-e,$ and see whether one
can prove that $R$ embeds in the resulting ring.
If we had a nice normal form for elements of $R,$ then
the Diamond Lemma~\cite{<>} might be used to get a
normal form for the extended ring, allowing us to see whether
the natural map of $R$ into it is indeed an embedding.

The Fitting Decomposition Theorem \cite[p.\,299 (19.16)]{TYL_1st}
implies that the endomorphism ring of any module of
{\em finite length} over any ring is strongly clean.
Such endomorphism rings are Dedekind-finite,
but conceivably, that theorem might be useful
in some multi-step construction of a strongly
clean non-Dedekind-finite ring.

For some results relating strong cleanness to
other ring-theoretic conditions, including
Dedekind finiteness, see~\cite{C+D+N}.

\section{Some related topics}\label{S.also}

An element $x$ of a ring $R$ is called {\em uniquely clean}
if there is a {\em unique} decomposition $x=e+u$ with $e$
idempotent and $u$ invertible, and {\em uniquely
strongly clean} if there is a unique such decomposition
in which $e$ and $u$ commute with each other.
Rings in which all elements have these properties
are called {\em uniquely clean rings,}
respectively, {\em uniquely strongly clean rings}.
There has been considerable research on these classes
of rings, e.g.,~\cite{C+W+Z}, \cite{WKN+YZ},~\cite{Y+L}.

An open question,
\cite[Question~19]{C+W+Z}, \cite[Question~5.1]{Y+L},
is whether every homomorphic image of a uniquely strongly
clean ring is again uniquely strongly clean.
This is known to be true for uniquely strongly clean rings
in which all idempotents are central (which are precisely
the uniquely clean rings.
For that equivalence, see~\cite[Lemma~4]{WKN+YZ},
and for the result on homomorphic images,~\cite[Theorem~22]{WKN+YZ}).
As a possible way to look for a counterexample if idempotents
are not required to be central,
one might first note that though the matrix ring
$M_2(\Z/2\Z)$ is strongly clean, it is not uniquely strongly
clean, since $x=e_{12}+e_{21}+e_{22}$ has the two strongly
clean decompositions $0+x$ and $1+(x-1),$ and then
try to find a uniquely strongly clean ring
$R,$ and a surjective homomorphism $R\to M_2(\Z/2\Z)$ under which
distinct elements $p$ and $q$ both map to $x,$
with the (unique) strongly clean decompositions of $p$ and $q$
mapping to the distinct decompositions of that common image.

Another topic:
Observe that a nonzero uniquely strongly clean ring $R$ can never be
an algebra over a field $k$ with more than two elements, since
if $u$ is an element of $k$ other than $0$ and $1,$
it has the two strongly clean decompositions
$0+u$ and $1+(u-1).$
Even if a uniquely strongly clean ring $R$ is not assumed an
algebra over a field, image rings of characteristic~$2$ occur
throughout the study of such rings.
For example, the quotient of $R$ by its Jacobson
radical is a Boolean ring \cite[Theorem~20,\,(1)$\implies$(3)]{WKN+YZ}.

One might get interesting results not restricted to rings
with homomorphic images of characteristic $2$
if,
in an algebra $R$ over a field $k,$ one defined
a ``metaidempotent'' ({\em is there an existing term?})
to mean a $\!k\!$-linear combination $r$ of a family of
mutually commuting idempotents; equivalently,
an element $r$ that satisfies a polynomial relation
of the form $(r-a_1)\dots (r-a_n)=0,$ where the
$a_i$ are {\em distinct} members of $k;$ and one might
study $\!k\!$-algebras $R$ which satisfy the generalization
of unique strong cleanness saying that every element has a
unique decomposition as the sum of a metaidempotent
and a member of the Jacobson radical $J(R)$
(cf.\ last sentence of preceding paragraph).
There might, again, be versions of this condition for rings which
are not themselves algebras over fields, perhaps with $R/J(R)$
an algebra over a field.

We end by noting a result of a similar flavor
to the existence result we tried unsuccessfully
to prove in section~\ref{S.thoughts},
but which is, in contrast, trivial to prove.

\begin{lemma}\label{L.unit+unit}
Every ring $R$ can be embedded in a ring $R'$ such that
every element of $R'$ is a sum of two commuting units.
Namely, the formal Laurent series ring $R'=R((t))=R[[t]][t^{-1}]$ has
this property.
\end{lemma}

\begin{proof}
Given $x\in R',$ choose an integer $N$ such that $x\in t^{N+1}R[[t]].$
Let $u=t^N\!+x$ and $u'= -t^N.$
Since $u$ and $u'$ both have invertible leading terms, they are
both units.
They commute, since $t$ is central; and clearly~$u+u'=x.$
\end{proof}

\section{Acknowledgements}\label{S.ackn}
I am indebted to Tom Dorsey and T.\,Y.\,Lam for pointing out
the question from \cite{WKN} that led to Proposition~\ref{P.yx_neq_1},
and to Pace Nielsen and T.\,Y.\,Lam for helpful
comments on earlier versions of this note.


\begin{thebibliography}{00}

\bibitem{<>} George M. Bergman,
{\em The diamond lemma for ring theory,}
Advances in Mathematics {\bf 29} (1978) 178--218.
MR0506890.

\bibitem{C+D+N} Victor Camillo, Thomas J. Dorsey and Pace P. Nielsen,
{\em Dedekind-finite strongly clean rings,}
Comm. Algebra {\bf 42} (2014) 1619-1629.
MR3169656


\bibitem{C+W+Z} Jianlong Chen, Zhou Wang and Yiqiang Zhou,
{\em Rings in which elements are uniquely the sum of an idempotent
and a unit that commute},
J. Pure Appl. Algebra {\bf 213} (2009) 215-223.
MR2467398

\bibitem{TYL_1st} T. Y. Lam,
{\em A first course in noncommutative rings,}
second edition, Springer GTM, 131.
2001. xx+385~pp.
MR1838439


\bibitem{WKN} W. K. Nicholson,
{\em Strongly clean rings and Fitting's lemma,}
Comm. Algebra {\bf 27} (1999) 3583--3592.
MR1699586

\bibitem{WKN+YZ} W. K. Nicholson and Y. Zhou, 
{\em Rings in which elements are uniquely the sum of an
idempotent and a unit},
Glasg. Math. J. {\bf 46} (2004) 227-236.
MR2062606


\bibitem{Y+L}
Zhiling Ying
and T. Y. Lam,
{\em A new study of strongly clean elements in rings,}
preprint, 11\,pp.


\end{thebibliography}
\end{document}